\documentclass[a4paper,twoside]{article}
\pdfoutput=1
\usepackage{epsfig}
\usepackage{subfigure}
\usepackage{calc}
\usepackage{amssymb}
\usepackage{amstext}
\usepackage{amsmath}
\usepackage{amsthm}
\usepackage{multicol}
\usepackage{pslatex}
\usepackage{apalike}

\usepackage{hyperref}
\hypersetup{
    colorlinks=true,
    linkcolor=black,
    citecolor=black,
    filecolor=black,
    urlcolor=black,
}

\usepackage{graphicx}

\usepackage{mathrsfs}
\usepackage{verbatim}
\usepackage{url}

\DeclareMathOperator*{\mathsup}{sup}
\DeclareMathOperator\erf{erf}

\newcommand{\RR}{\ensuremath{\mathbb{R}}}

\theoremstyle{definition}

\allowdisplaybreaks

\newtheorem{thm}{Theorem}[section]

\newtheorem{cor}[thm]{Corollary}
\newtheorem{prop}[thm]{Proposition}

\theoremstyle{definition}

\theoremstyle{remark}
\newtheorem{rem}[thm]{Remark}

\usepackage{SCITEPRESS}     

\begin{document}

\title{Quantitative Assessment of Robotic Swarm Coverage} 

\author{\authorname{Brendon G. Anderson\sup{1}, Eva Loeser\sup{2}, Marissa Gee\sup{3}, Fei Ren\sup{1}, Swagata Biswas\sup{1}, \\Olga Turanova\sup{1}, Matt Haberland\sup{1*}, and Andrea L. Bertozzi\sup{1}}
\affiliation{\sup{1}UCLA, Dept. of Mathematics, Los Angeles, CA 90095}
\affiliation{\sup{2}Brown University, Mathematics Department, Providence, RI 02912}
\affiliation{\sup{3}Harvey Mudd College, Dept. of Mathematics, Claremont, CA 91711}
\email{\sup{*} Corresponding author: haberland@ucla.edu}
}


\abstract{This paper studies a generally applicable, sensitive, and intuitive error metric for the assessment of robotic swarm density controller performance. Inspired by vortex blob numerical methods, it overcomes the shortcomings of a common strategy based on discretization, and unifies other continuous notions of coverage. We present two benchmarks against which to compare the error metric value of a given swarm configuration: non-trivial bounds on the error metric, and the  probability density function of the error metric when robot positions are sampled at random from the target swarm distribution. We give rigorous results that this probability density function of the error metric obeys a central limit theorem, allowing for more efficient numerical approximation. For both of these benchmarks, we present supporting theory, computation methodology, examples, and MATLAB implementation code.}

\onecolumn \maketitle \normalsize \vfill

\section{\uppercase{Introduction}}
\label{sec:introduction}

\noindent Much of the research in swarm robotics has focused on determining control laws that elicit a desired group behavior from a swarm \cite{brambilla2013swarm}, while less attention has been placed on methods for quantifying and evaluating the performance of these controllers. Both \cite{brambilla2013swarm} and \cite{cao1997cooperative}  point out the lack of developed performance metrics for assessing and comparing swarm behavior, and \cite{brambilla2013swarm} notes that when performance metrics are developed, they are often too specific to the task being studied to be useful in comparing performance across controllers. This paper develops an error metric that evaluates one common desired swarm behavior: distributing the swarm according to a prescribed spatial density.

In many applications of swarm robotics, the swarm must spread across a domain according to a target distribution in order to achieve its goal. Some examples are in surveillance and area coverage \cite{bruemmer2002robotic,hamann2006analytical,howard2002mobile,schwager2006distributed}, achieving a heterogeneous target distribution \cite{elamvazhuthi2016coverage,berman2011design,demir2015decentralized,shen2004hormone,elamvazhuthi2015optimal}, and aggregation and pattern formation \cite{soysal2006macroscopic,spears2004distributed,reif1999social,sugihara1996distributed}. Despite the importance of assessing performance, some studies such as \cite{shen2004hormone} and \cite{sugihara1996distributed} rely only on qualitative methods such as visual comparison. Others present performance metrics that are too specific to be used outside of the specific application, such as measuring cluster size in \cite{soysal2006macroscopic}, distance to a pre-computed target location in \cite{schwager2006distributed}, and area coverage by tracking the path of each agent in \cite{bruemmer2002robotic}. In \cite{reif1999social}  an $L^2$ norm of the difference between the target and achieved swarm densities is considered, but the notion of achieved swarm density is particular to the controllers under study.

We develop and analyze an error metric that quantifies how well a swarm achieves a prescribed spatial distribution. Our method is independent of the controller used to generate the swarm distribution, and thus has the potential to be used in a diverse range of robotics applications.  
In \cite{li2017decentralized} and \cite{zhang2017performance}, error metrics similar to the one presented here are used, but their properties are not discussed in sufficient detail for them to be widely adopted. 
In particular, although the error metric that we study always takes values somewhere between 0 and 2, these values are, in general, not achievable for an arbitrary desired distribution and a fixed number of robots. How then, in general, is one to judge whether the value of the error metric, and thus the robot distribution, achieved by a given swarm control law is ``good'' or not? We address this by studying two benchmarks, 
\begin{enumerate}
\item the \emph{global extrema} of the error metric, and
\item the \emph{probability density function} (PDF) of the error metric  when robot positions are sampled from the target distribution,
\end{enumerate}
which were first proposed in \cite{li2017decentralized}. Using tools from nonlinear programming for (1) and rigorous probability results for (2), we put each of these benchmarks on a firm foundation.  In addition, we provide MATLAB code for performing all calculations at \url{https://git.io/v5ytf}. Thus, by using the methods developed here, one can assess the performance of a given controller for any target distribution by comparing the error metric value of robot configurations produced by the controller against benchmarks (1) and (2).

Our paper is organized as follows. Our main definition, its basic properties, and a comparison to common discretization methods is presented in Section \ref{sec:error}. Then,  Section \ref{sec:extrema} and Section \ref{sec:errorpdf} are devoted to studying (1) and (2), respectively.  
We suggest future work and conclude in Section \ref{sec:conclusion}.

\section{\uppercase{Quantifying Coverage}}
\label{sec:error}

\noindent One difficulty in quantifying swarm coverage is that the target density within the domain is often prescribed as a continuous (or piecewise continuous) function, yet the swarm is actually composed of a finite number of robots at discrete locations. A common approach for comparing the actual positions of robots to the target density function begins by discretizing the domain (e.g. \cite{berman2011design,demir2015decentralized}). We demonstrate the pitfalls of this in Subsection \ref{sec:naive}. Another possible route (the one we take here) is to use the robot positions to construct a continuous function that represents coverage (e.g. \cite{HS,ZC,ayvali2017ergodic}). It is also possible to use a combination of the two methods, as in \cite{Cortes}. 

The method we present and analyze is inspired by vortex blob numerical methods for the Euler equation and the aggregation equation (see \cite{CB} and the references therein). A similar strategy was used in \cite{li2017decentralized} and \cite{zhang2017performance} to measure the effectiveness of a certain robotic control law, but to our knowledge, our work is the first study of any such  method in a form sufficiently general for common use. 

\subsection{Definition}

We are given a bounded region\footnote{We present our definitions for any number of dimensions $d\geq 1$ to demonstrate their generality. However, in  the latter sections of the paper, we restrict ourselves to $d=2$, a common setting in ground-based applications.}  
$\Omega \in \RR^d$, a desired robot distribution $\rho:\Omega \rightarrow (0, \infty)$ satisfying $\int_\Omega \rho(z)\, dz=1$, and $N$ robot positions $x_1, ..., x_N\in \Omega$.  
To compare the discrete points $x_1, ... x_N$ to the function $\rho$, we place a ``blob" of some shape and size at each point $x_i$. The shape and size parameters have two  physical interpretations as:
\begin{itemize}
\item the area in which an individual robot performs its task, or
\item inherent uncertainty in the robot's position.
\end{itemize}
Either of these interpretations (or a combination of the two) can be invoked to make a meaningful choice of these parameters.  

The blob or \emph{robot blob function} can be any function $K(z):\RR^d \rightarrow \RR$ that  is non-negative on $\Omega$ and satisfies $\int_{\RR^d}K(z)dz = 1$. While not required, it is often natural to use a $K$ that is radially symmetric and decreasing along radial directions. In this case, we need one more parameter, a positive number $\delta$, that controls how far the effective area (or inherent positional uncertainty) of the robot extends. We then define $K^\delta$ as,
\begin{equation}
\label{def:Kdelta}
K^\delta(z)=\frac{1}{\delta^d} K\left(\frac{z}{\delta}\right).
\end{equation}
We point out that we still have $\int_{\RR^d} K^\delta(z)\, dz=1$. One choice of $K^\delta$ would be a scaled indicator function, for instance, a function of constant value within a disc of radius $\delta$ and $0$ elsewhere. This is an appropriate choice when a robot is considered to either perform its task at a point or not, and where there is no notion of the degree of its effectiveness. For the remainder of this paper, however, we usually take $K$ to be the Gaussian 
\[
G(z)=\frac{1}{2\pi}\exp\left(-\frac{|z|^2}{2}\right),
\]
which is useful when the robot is most effective at its task locally and to a lesser degree some distance away. To define the \emph{swarm blob function} $\rho_N^\delta$, we place a blob $G^\delta$ at each robot position $x_i$, sum over $i$ and renormalize,  yielding,
\begin{equation}
\label{eqn:normalized_blob_function}
\rho_N^\delta(z;x_1,..., x_N) = \frac{\sum_{i=1}^N G^\delta(z-x_i)}{\sum_{i=1}^N \int_\Omega G^\delta(z-x_i)\, dz}.
\end{equation}
For brevity, we usually write $\rho_N^\delta(z)$ to mean $\rho_N^\delta(z;x_1,..., x_N)$. 
This swarm blob function gives a continuous representation of how the discrete robots are distributed. 
Note that each integral in the denominator of (\ref{eqn:normalized_blob_function}) approaches $1$ if $\delta$ is small or all robots are far from the boundary, so that we have,
\begin{equation}
\label{eqn:rho with N}
\rho_N^\delta(z) \approx \frac{1}{N}\sum_{i=1}^N G^\delta(z - x_i).
\end{equation}

We now define our notion of error, which we refer to as the \emph{error metric}: 
\begin{equation}
\label{eqn:instantaneous_error_metric}
e_N^\delta(x_1,...,x_N)=\int_\Omega \left| \rho_N^\delta(z) - \rho(z) \right| dz .
\end{equation}
We often write this as $e_N^\delta $ for brevity. 

\subsection{Remarks and Basic Properties}
\label{sec:error_properties}
Our error is  defined as the $L^1$ norm between the swarm blob function and the desired robot distribution $\rho$. One could use another $L^p$ norm; however, $p=1$ is a standard choice in applications that involve particle transportation and coverage such as \cite{zhang2017performance}. Moreover, the $L^1$ norm has a key property:  for any two integrable functions $f$ and $g$, 
\[
\int_\Omega\left|f-g\right| \, dz = 2\mathsup_{B\subset \Omega} \left|\int_B f \, dz - \int_B g \,  dz\right|.
\]
The other $L^p$ norms do not enjoy this property \cite[Chapter 1]{DG}. Consequently, by measuring $L^1$ norm on $\Omega$, we are also bounding the error we make on any particular subset, and, moreover, knowing the error on ``many" subsets gives an  estimate of the total error. This means that by using the $L^1$ norm we capture the idea  that discretizing the domain provides a measure of error, but avoid the pitfalls of discretization methods described in Subsection \ref{sec:naive}.

Studies in optimal control of swarms often use the $L^2$ norm due to the favorable inner product structure \cite{zhang2017performance}. We point out that the $L^1$ norm is bounded from above by the $L^2$ norm due to the Cauchy-Schwarz inequality and the fact that $\Omega$ is a bounded region. Thus, if an optimal control strategy controls the $L^2$ norm, then it will also control the error metric we present here.

Last, we note:
\begin{prop}
For any $\Omega$, $\rho$, $\delta$, $N$, and $(x_1,...,x_N)$, 
\[
0\leq e_N^\delta \leq 2.
\]
\end{prop}
\begin{proof} This follows directly from the basic property 
\[
\int |f | \, dz - \int |g| \, dz \leq \int|f-g| \, dz \leq \int |f| \, dz +\int |g| \, dz 
\]
and our normalization.
\end{proof} 
The theoretical minimum of $e^\delta_N$  can only be approached for a general target distribution when $\delta$ is small and $N$ is large, or in the trivial special case when the target distribution is exactly the sum of $N$ Gaussians of the given $\delta$, motivating the need to develop benchmarks (1) and (2).

\subsection{Variants of the Error Metric}
The notion of error defined by (\ref{eqn:instantaneous_error_metric}) is suitable for tasks that require good instantaneous coverage. For tasks that involve tracking average coverage over some period of time (and in which the robot positions are functions of time $t$), an alternative ``cumulative'' version of the error metric is
\begin{equation}
\label{eqn:cumulative_error_metric}
\int_\Omega \left| \frac{1}{M}\sum_{j = 1}^M{\rho_N^\delta(z,t_j)} - \rho_\Omega(z) \right| dz
\end{equation}
for time points $j = 1, \dots, M$. This is a practical, discrete-time version of the metric used in \cite{zhang2017performance}, which uses a time integral rather than a sum, as in practice, position measurements can only be made at discrete times. While this cumulative error metric is, in general, distinct from the instantaneous version of (\ref{eqn:instantaneous_error_metric}), note that the extrema and  PDF of this cumulative version can be calculated as the extrema and PDF of the instantaneous error metric  with $MN$ robots. Therefore, in subsequent sections we restrict our attention to the extrema and PDF of the instantaneous formulation without loss of generality.

In addition, \cite{zhang2017performance} considers a one-sided notion of error, in which a scarcity of robots is penalized but an excess is not, that is,
\[
\hat{e}_N^\delta = \int_{\Omega^-} \left| \rho^\delta_N(z) - \rho(z) \right| \, dz, 
\]
where $\Omega^-:=\{z|\rho_N^\delta(z)\leq \rho(z)\}$. Remarkably, $\hat{e}_N^\delta$ and $e_N^\delta$ are related by:
\begin{prop}$e_N^\delta = 2\hat{e}_N^\delta$. \label{prop:hat}
\end{prop}
\begin{proof}
Let $\Omega^+ =  \Omega \setminus \Omega^-$. Since $\Omega = \Omega^-\cup \Omega^+$, we have,
\begin{align*}
\int_{\Omega^-}\rho_N^\delta\, dz &+ \int_{\Omega^+}\rho_N^\delta\, dz = \int_\Omega \rho_N^\delta \, dz= 1 =\\&=
\int_\Omega \rho \, dz= \int_{\Omega^-}\rho\, dz+\int_{\Omega^+}\rho\, dz.
\end{align*}
Rearranging and taking absolute values we find
\[
\int_{\Omega^-} \left| \rho_N^\delta - \rho \right| \, dz= \int_{\Omega^+} \left| \rho_N^\delta - \rho \right|\, dz,
\]
as each integrand is of the same sign everywhere within the limits of integration. We notice that the left-hand side and therefore the right-hand side of the previous line equal $\hat{e}_N^\delta$. On the other hand, their sum equals $e^\delta_N$. Thus our claim holds.  
\end{proof}

The definition of $\hat{e}_N^\delta$ is particularly useful in conjunction with the choice of $K^\delta$ as a scaled indicator function, as $\hat{e}_N^\delta$ becomes a direct measure of the deficiency in coverage of a robotic swarm. For instance, given a swarm of surveillance robots, each with observational radius $\delta$, $\hat{e}_N^\delta$ is the percentage of the domain not observed by the swarm.\footnote{The notion of ``coverage'' in \cite{bruemmer2002robotic} might be interpreted as $\hat{e}_N^\delta$ with $\delta$ as the width of the robot. There, only the time to complete coverage ($\hat{e}_N^\delta = 0$) was considered.} Proposition \ref{prop:hat} implies that $\frac{1}{2}e_N^\delta$ also enjoys this interpretation.

\subsection{Calculating $e_N^\delta$}
In practice, the integral in (\ref{eqn:instantaneous_error_metric}) can rarely be carried out analytically, primarily because the integral needs to be separated into regions for which the quantity $\rho_N^\delta(z) - \rho(z)$ is positive and regions for which it is negative, the boundaries between which are usually difficult to express in closed form. We find that a simple generalization of the familiar rectangle rule converges linearly in dimensions $d\leq 3$ and expect Monte Carlo and Quasi-Monte Carlo methods to produce a reasonable estimate in higher dimensions\footnote{We look forward to a physical swarm of robots being deployed -- and these results employed -- in four dimensions!}. More advanced quadrature rules can be used in low dimensions, but may suffer in accuracy due to nonsmoothness in the target distribution and/or stemming from the absolute value taken within the integral. 

\subsection{The Pitfalls of Discretization}
\label{sec:naive}
We conclude this section by analyzing a  measure of error that involves discretizing the domain. In particular, we  show in Propositions \ref{prop:disc1} and \ref{prop:disc} that the values produced by this method are strongly dependent on a choice of discretization. 
In particular, this error approaches its theoretical minimum when the discretization is too coarse and its theoretical maximum when the discretization is too fine, regardless of robot positions.

Discretizing the domain means dividing  $\Omega$  into $M$ disjoint regions $\Omega_i \subset \Omega$ such that $\bigcup_{i=1}^{M}\Omega_i = \Omega$. 
Within each region, the desired proportion of robots is the integral of the target density function within the region $\int_{\Omega_i} \rho(z) dz$. Using $N_i$ to denote the observed number of robots in $\Omega_i$, we can define an  error metric as
\begin{equation}
\label{eqn:naive_error}
\mu = \sum_{i=1}^M \left| \int_{\Omega_i} \rho(z)\, dz - \frac{N_i}{N} \right|.
\end{equation}
It is easy to check that $0\leq \mu\leq 2$ always holds. 
One advantage of this approach is that $\mu$ is very easy to compute, but there are two major drawbacks. 

\subsubsection{Choice of Domain Discretization}
The choice for domain discretization is not unique, and this choice can dramatically affect the  value of $\mu$, as demonstrated by the following two propositions.

\begin{prop} 
\label{prop:disc1}
If $M=1$ then $\mu=0$.
\end{prop}
\begin{proof}
When $M=1$, (\ref{eqn:naive_error}) becomes
\[
\mu= \left|\int_\Omega \rho(z)\, dz - 1 \right|=0. 
\]
\end{proof}
The situation of perfectly fine discretization is in complete contrast.
\begin{prop}
\label{prop:disc}
Suppose  the robot positions are distinct\footnote{This is reasonable in practice as two physical robots cannot occupy the same point in space. In addition, the proof can be modified to produce the same result even if the robot positions coincide.} and  the regions $\Omega_i$ are sufficiently small such that, for each $i$, $\Omega_i$  contains at most one robot and   $ \int_{\Omega_i}\rho(z) \, dz \leq 1/N$ holds. Then $\mu\rightarrow 2$ as $|\Omega_i|\rightarrow 0$. 
\end{prop}
\begin{proof}
Let us relabel the $\Omega_i$ so that for $i=1,..., M-N$ there is no robot in $\Omega_i$, and thus each of the $\Omega_i$ for $i=M-N+1, ..., M$ contains exactly one robot. In this case, the expression for error $\mu$ becomes,
\begin{equation}
\label{eqn:naive_error_fine_grid}
\mu = \sum_{i=1}^{M-N} \int_{\Omega_i} \rho \, dz + \sum_{i=M-N+1}^{M}\left| \int_{\Omega_i}\rho \, dz - \frac{1}{N}\right|.
\end{equation}
Since $ \int_{\Omega_i}\rho \, dz \leq 1/N$ holds, then with the identity 
\begin{align*}
\sum_{i=1}^{M-N} \int_{\Omega_i} \rho  \, dz = 1- \sum_{i=M-N+1}^{M} \int_{\Omega_i}\rho \, dz,
\end{align*}
we can rewrite (\ref{eqn:naive_error_fine_grid}) as,
\begin{align*}
\mu = 1- &\sum_{i=M-N+1}^{M} \int_{\Omega_i}\rho \, dz  + \sum_{i=M-N+1}^{M}\left( \frac{1}{N} - \int_{\Omega_i}\rho \, dz \right) \\
&= 2-2 \sum_{i=M-N+1}^{M} \int_{\Omega_i}\rho \, dz.
\end{align*}
Thus $\mu\rightarrow 2$ as $M \rightarrow \infty$ and $|\Omega_i|\rightarrow 0$. 
\end{proof}


Note that the shape of each region is also a choice that will affect the calculated value. While our  approach  also requires the choice of some size and shape (namely, $\delta$ and $K$), these parameters  have much more immediate physical interpretation, making appropriate choices easier to make.

\subsubsection{Error Metric Discretization and Desensitization}
Perhaps more importantly, by discretizing the domain, we also discretize the range of values that the the error metric can assume. While this may not be inherently problematic, we have simultaneously desensitized the error metric to changes in robot distribution within each region. That is, as long as the number of robots $N_i$ within each region $\Omega_i$ does not change, the distribution of robots within any and all $\Omega_i$ may be changed arbitrarily without affecting the value of $\mu$. On the other hand, the error metric $e_N^\delta$ is continuously sensitive to differences in distribution.

\section{\uppercase{Error Metric Extrema}}
\label{sec:extrema}

%

\noindent In the rest of the paper, we provide tools for determining whether or not the values of $e_N^\delta$ produced by a controller in a given situation are ``good". As mentioned in Section \ref{sec:error_properties}, it is simply not possible to achieve $e_N^\delta = 0$ for every combination of target distribution $\rho$, number of robots $N$, and blob size $\delta$ . Therefore, we would like to compare the achieved value of $e_N^\delta$ against its \emph{realizable} extrema given $\rho$, $N$, and $\delta$. But $e_N^\delta$ is a highly nonlinear function of the robot positions $(x_1, ..., x_N)$,
and trying to find its extrema analytically has been intractable. Thus, we approach this problem by using nonlinear programming. 

\subsection{Extrema Bounds via Nonlinear Programming}
Let $x=(x_1,..., x_N)$ represent a vector of $N$ robot coordinates. The optimization problem is
\begin{align}
\label{Optimization Problem}
	&\text{minimize } e_N^{\delta}(x_1,..., x_N),   \\
	&\text{subject to } x_i \in \Omega \text{ for } i \in \{1,2,\hdots,N\}. \notag
\end{align}
Note that the same problem structure can be used to find the maximum of the error metric by minimizing $-e_N^{\delta}$. Given $\rho$, $N$, and $\delta$, we solve these problems using a standard nonlinear programming solver, MATLAB's \texttt{fmincon}. 

A limitation of all general nonlinear programming algorithms is that successful termination produces only a local minimum, which is not guaranteed to be the global minimum. There is no obvious re-formulation of this problem for which a global solution is guaranteed, so the best we can do is to use a local minimum produced by nonlinear programming as an upper bound for the minimum of the error metric. Heuristics, such as multi-start (running the optimization many times from several initial guesses and taking the minimum of the local minima) can be used to make this bound tighter. This bound, which we call $e^-$, and the equivalent bound on the maximum, $e^+$, serve as benchmarks against which we can compare an achieved value of the error metric. This is reasonable, because if a configuration of robots with a lower value of the error metric exists but eludes numerical optimization, it is probably not a fair standard against which to compare the performance of a general controller.

\subsection{Relative Error}
The performance of a robot distribution controller can be quantitatively assessed by calculating the error value $e_{\text{observed}}$ of a robot configuration it produces, and comparing this value against the extrema bounds $e^-$ and $e^+$. If the robot positions $x_1,..., x_N$ produced by a given controller are constant, then $e_{\text{observed}}$ can simply be taken as $e_N^\delta(x_1,..., x_N)$. In general, however, the positions  $x_1,..., x_N$ may change over time. In this case, we suggest using the third-quartile value observed after the system reaches steady state, which we denote $e_\text{Q3}$.

Consider the relative error
\begin{equation*}
\label{eqn:erel}
e_\text{rel} = \frac{e_{\text{observed}} - e^-}{e^+ - e^-}.
\end{equation*}
We suggest that if $e_\text{rel}$ is less than 10\%, the performance of the controller is quite close to the best possible, whereas if this ratio is 30\% or higher, the performance of the controller is rather poor. 

%
%
%

\subsection{Example}
\label{example optimization}
We apply this method to assess the performance of the controller in \cite{li2017decentralized}, which guides a swarm of $N=200$ robots with $\delta = 2$in (the physical radius of the robots) to achieve a ``ring distribution"\footnote{The ring distribution $\rho_{\text{ring}}$ is defined on the Cartesian plane with coordinates $z=(z_1$,$z_2)$ as follows. Let inner radius $r_1 = 11.4$in, outer radius $r_2 = 20.6$in, width $w = 48$in, height $h=70$in, and $\rho_0 = 2.79 \times 10^{-5}$. Let domain $\Omega = \{z: z_1 \in [0, w], z_2 \in [0, h]\}$ and region $\Gamma = \{z:r^2_1 < (z_1 - \frac{w}{2})^2 + (z_2 - \frac{h}{2})^2 < r^2_2\}$. Then $\rho_{\text{ring}}(z) = 36\rho_0 \text{ if } z \in \Omega \cap \Gamma$, $\rho_0 \text{ if } z \in \Omega \setminus \Gamma$.}.

Under this stochastic control law, the behavior of the error metric over time appears to be a noisy decaying exponential. Therefore, we fit to the data shown in Figure 7 of \cite{li2017decentralized} a function of the form $f(t) = \alpha + \beta \exp(-\frac{t}{\tau})$ by finding error asymptote $\alpha$, error range $\beta$, and time constant $\tau$ that minimize the sum of squares of residuals between $f(t)$ and the data. By convention, the steady state settling time is taken to be $t_\text{s} = 4 \tau$, which can be interpreted as the time at which the error has settled to within $2\%$ of its asymptotic value \cite{barbosa2004tuning}. The third quartile value of the error metric for $t > t_\text{s}$ is $e_{Q3} = 0.5157$.

To determine an upper bound on the global minimum of the error metric, we computed 50 local minima of the error metric starting with random initial guesses, then took the lowest of these to be $e^- = 0.28205$. An equivalent procedure bounds the global maximum as $e^+ = 1.9867$, produced when all robot positions coincide near a corner of the domain. The corresponding swarm blob functions are depicted in Figures \ref{fig:error_metric_minimum} and \ref{fig:error_metric_maximum}. 
Note that the minimum of the error metric is significantly higher than zero for this finite number of robots of nonzero radius, emphasizing the importance of performing this benchmark calculation rather than using zero as a reference.

\begin{figure}[ht]
\includegraphics[width=\linewidth]{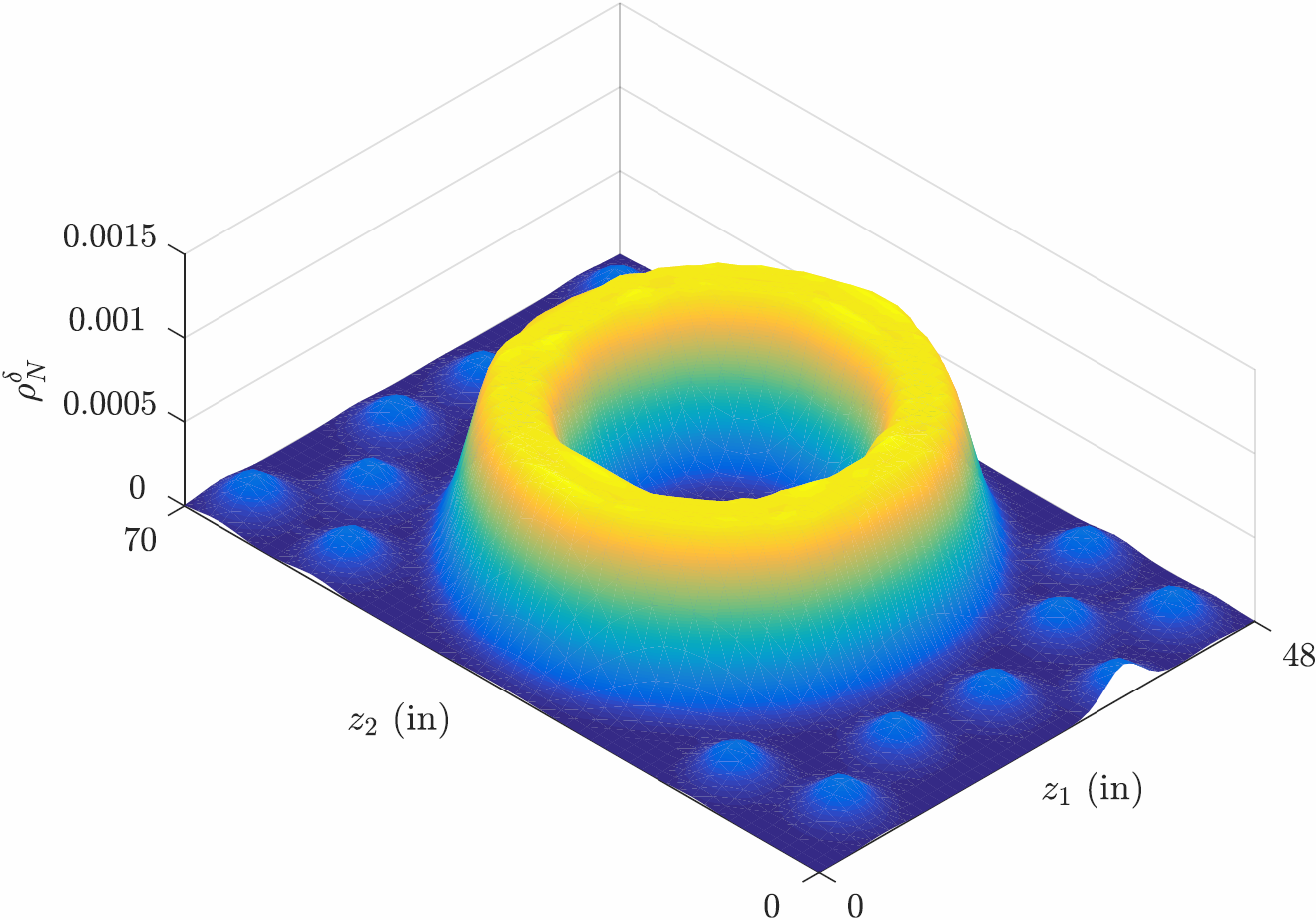}
\caption{Swarm blob function $\rho_{N=200}^{\delta = 2 \text{in}}$ corresponding with the robot distribution that yields a minimum value of the error metric for the ring distribution, 0.28205.}
\label{fig:error_metric_minimum}
\end{figure}

\begin{figure}[ht]
\includegraphics[width=\linewidth]{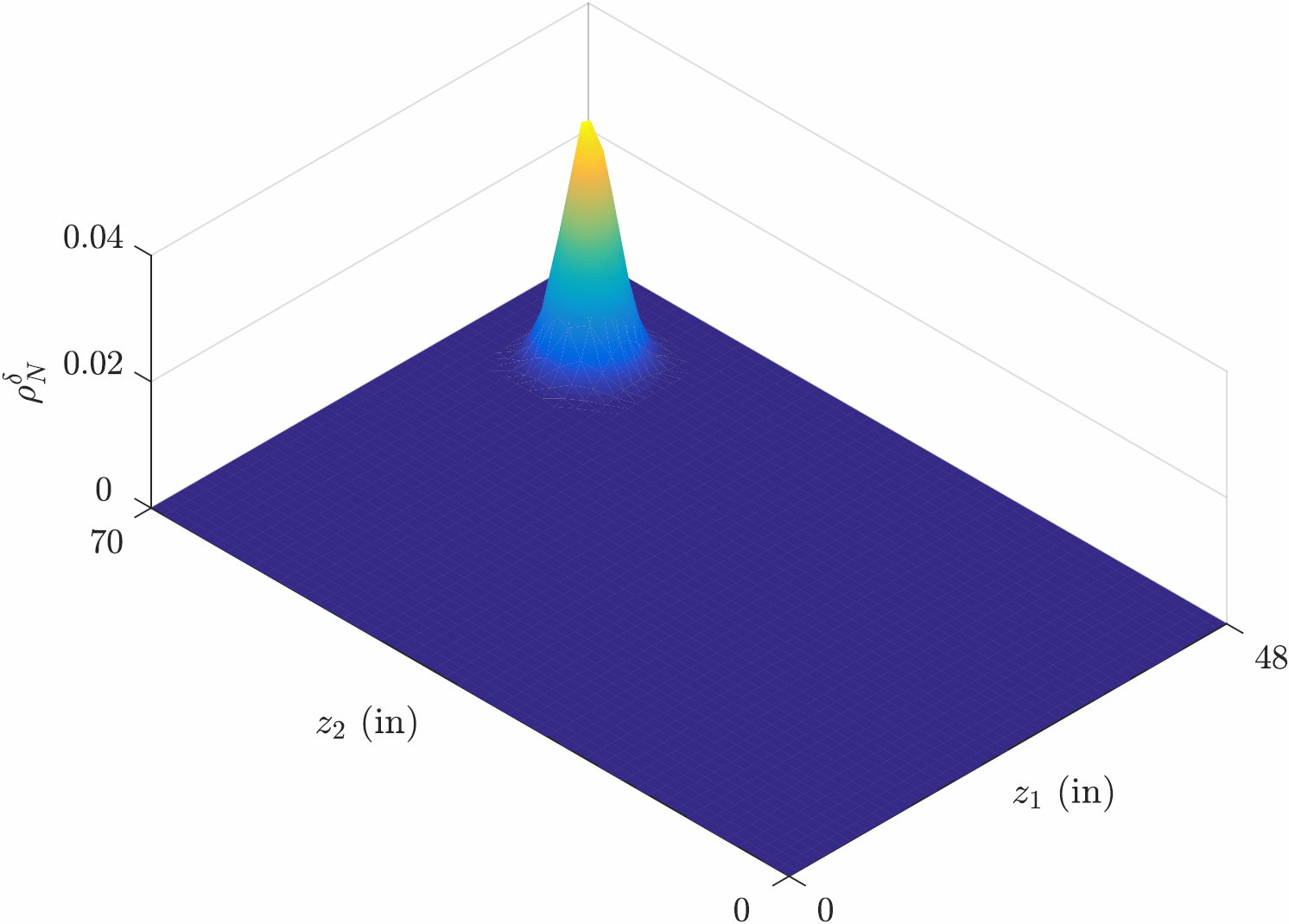}
\caption{Swarm blob function $\rho_{N=200}^{\delta = 2 \text{in}}$ corresponding with the robot distribution that yields a maximum value of the error metric for the ring distribution, 1.9867. This occurs when all robots coincide outside the ring.}
\label{fig:error_metric_maximum}
\end{figure}

Using these values for $e_\text{Q3}$, $e^-$, and $e^+$, we calculate $e_\text{rel}$ according to Equation \ref{eqn:erel} as $13.71\%$.

While the sentiment of the $e_\text{rel}$ benchmark is found in \cite{li2017decentralized}, we have made three important improvements to the calculation to make it suitable for general use. First, in \cite{li2017decentralized} values analogous to $e^-$ and $e^+$ were found by ``manual placement'' of robots, whereas we have used nonlinear programming so that the calculation is objective and repeatable. Second, \cite{li2017decentralized} refers to steady state but does not suggest a definition. Adopting the 2\% settling time convention not only allows for an unambiguous calculation of $e_{Q3}$ and other steady-state error metric statistics, but provides a metric for assessing the speed with which the control law effects the desired distribution.  Finally, \cite{li2017decentralized} uses the \emph{minimum observed value} of the error metric in the calculation, but we suggest that the third quartile value better represents the distribution of error metric values achieved by a controller, and thus is more representative of the controller's overall performance. 

These changes account for the difference between our calculated value of $e_\text{rel} = 13.71\%$ and the report in \cite{li2017decentralized} that the error is  ``$7.2$\% of the range between the minimum error value \dots and maximum error value''. Our substantially higher value of $e_\text{rel}$ indicates that the performance of this controller is not very close the best possible. We emphasize this to motivate the need for our second benchmark in Section \ref{sec:errorpdf}, which is more appropriate for a stochastic controller like the one presented in \cite{li2017decentralized}.

\subsection{Error Metric for Optimal Swarm Design}
So far we have taken $N$ and $\delta$ to be fixed; we have assumed that the robotic agents and size of the swarm have already been chosen. We briefly consider the use of the error metric as an objective function for the \emph{design} of a swarm. Simply adding $\delta > 0$ as a decision variable to (\ref{Optimization Problem}) and solving at several fixed values of $N$ provides insight into how many robots of what effective working radius are needed to achieve a given level of coverage for a particular target distribution. Visualizations of such calculations are provided in Figure \ref{fig:blobs} and the supplemental video.

\begin{figure*}
\includegraphics[width=\linewidth]{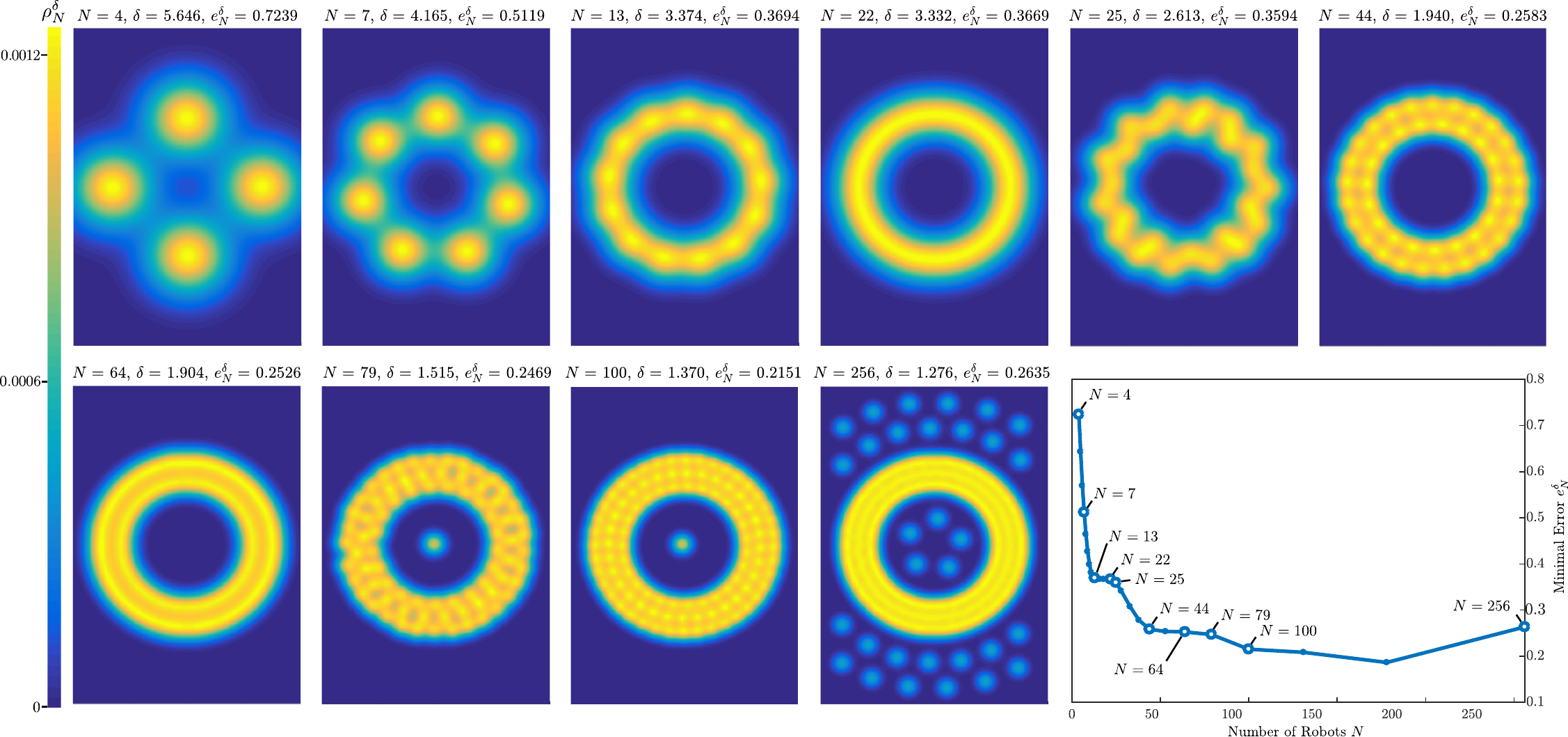}
\centering
\caption{Swarm blob functions $\rho_{N}^{\delta}$ corresponding with the robot distributions and values of $\delta$ that yield the minimum value of the error metric for the ring distribution target. Inset graph shows the relationship between $N$ and the minimum value of the error metric observed from repeated numerical optimization. Due to long execution time of optimization at $N = 256$, fewer local minima were calculated; this likely explain the rise in the minimal error metric value. This highlights the need in future work to find a more efficient formulation of this optimization problem or to use a more effective solver.}
\label{fig:blobs}
\end{figure*}

Note that `breakthroughs', or relatively rapid decreases in the error metric, can occur once a critical number of robots are available; these correspond with a qualitative change in the distribution of robots. For example, at $N=22$ the robots are arranged in a single ring; beginning with $N=25$ we see the robots begin to be arranged in two separate concentric rings of different radii and the error metric begins to drop sharply. On a related note, there are also ``lulls" in which increasing the number of robots has little effect on the minimum value of the error metric, such as between $N=44$ and $N=79$. Studies like these can help a swarm designer determine the best number of robots $N$ and effective radius of each $\delta$ to achieve the required coverage.

\section{\uppercase{Error Metric Probability Density Function}}
\label{sec:errorpdf}
\noindent In the previous section we have described how to find bounds on the minimum and maximum values for error. A question that remains is, how ``easy" or ``difficult" is it to achieve such  values? 
Answering this question is important in order to use the error metric to assess the effectiveness of an underlying control law. Indeed, a given control law --- especially a stochastic control law --- may tend to produce robot positions with error well above the minimum, and it is necessary to assess these values as well.

According to the setup of our problem, the goal of any such control law is for the robots to achieve the desired distribution $\rho$. Thus,  whatever the particular control law is, it is natural to compare its outcome to simply picking the robot positions at random from the target distribution $\rho$. In this section we consider the robot positions as being sampled directly from the desired distribution and study the statistical properties of the error metric in this situation, both analytically and numerically, and suggest how they may be used as a benchmark with which to evaluate the performance of a swarm distribution controller. 


We take the robots' positions $X_1$, \dots, $X_N$, to be independent, identically distributed bivariate random vectors in $\Omega\subset \RR^2$ with probability density function $\rho$. We place a blob of shape $K$ at each of the $X_i$ (previously we took $K$ to be the Gaussian $G$), so that the swarm blob function is,
\begin{equation}
\label{rho random}
\rho_N^\delta (z) = \frac{1}{N\delta^2}\sum_{i=1}^N K^\delta\left(z - X_i\right),
\end{equation}
where $K^\delta$ is defined by (\ref{def:Kdelta}). We point out that the right-hand side of (\ref{rho random}) is exactly that of (\ref{eqn:rho with N}) upon taking $K$ to be the Gaussian $G$ and the robot locations $x_i$ to be the randomly selected $X_i$. The error $e_N^\delta$ is now a random variable, the value of which depends on the particular realization of the robot positions $X_1$, \dots, $X_N$, but which has a well-defined probability density function (PDF) and cumulative distribution function (CDF). We denote the PDF and CDF by $f_{e_N^\delta}$ and $F_{e_N^\delta}$, respectively. The performance of a stochastic robot distribution controller can be quantitatively assessed by calculating the error values $e_N^\delta(x_1,..., x_N)$ it produces in steady state and comparing their distribution to $f_{e_N^\delta}$. 

In Subsection \ref{theory converge} we present rigorous results that show that the error metric has an approximately normal distribution in this case. As a corollary we obtain that the limit of this error is zero  as $N$ approaches infinity and $\delta$ approaches 0. Subsections \ref{subsec:num} and \ref{subsec: num ex} include a numerical demonstration of these results. In addition, in \ref{subsec: num ex}, we present an example calculation.

The theoretical results presented in the next subsection not only support our numerical findings, they also allow for faster computation. Indeed, if one did not already know that the error when robots are sampled randomly from $\rho$ has a normal distribution for large $N$, tremendous computation may be needed to get an accurate estimate of this probability density function. On the other hand, since the results we present prove that the error metric has a normal distribution for large $N$, we need only fit a Gaussian function to the results of relatively little computation. 


\subsection{Theoretical Central Limit Theorem}
\label{theory converge}
The expression (\ref{rho random}) is the so-called \emph{kernel density estimator} of $\rho$.  This arises in statistics, where  $\rho$ is thought of as unknown, and $\rho_N^\delta$ is considered as an approximation to $\rho$.  
It turns out that, under appropriate hypotheses, the $L^1$ error between $\rho$ and $\rho^\delta_N$ has a normal distribution with mean and variance that approach zero as $N$ approaches infinity. In other words, a central limit theorem holds for the error. For such a result to hold, $\delta$ and $N$ have to be compatible. Thus, for the remainder of this subsection $\delta$ will depend on $N$, and we display this as $\delta(N)$. 
We have,
\begin{thm}
\label{thm:H}
Suppose $\rho$ is continuously twice differentiable, $K$ is zero outside of some bounded region and radially symmetric. Then,  for $\delta(N)$ satisfying
\begin{equation}
\label{condition delta}
\delta(N)=O(N^{-1/6}) \text{ and } \lim_{N\rightarrow \infty}\delta(N) N^{1/4}=\infty,
\end{equation}
we have 
\[
e^{\delta(N)}_N \approx \mathcal{N}\left(\frac{e(N)}{N^{1/2}}, \frac{\sigma(N)^2\delta(N)^2}{N}\right),
\]
where $\sigma^2(N)$ and $e(N)$ are deterministic quantities that are bounded uniformly in $N$.\footnote{Here $\mathcal{N}(\mu, \sigma^2)$ denotes the normal random variable of mean $\mu$ and variance $\sigma^2$, and we use the notation $\approx$ to mean  that the difference of the quantity on the left-hand side and on the right-hand side converges to zero in the sense of distributions as $N\rightarrow\infty$.}
\end{thm}
\begin{proof}
This follows from Horv\'ath \cite[Theorem, page 1935]{horvath1991lp}. For the convenience of the reader, we record that the quantities $N$, $\delta$, $\rho$, $\rho^\delta_N$, $e_N^\delta$ that we use here correspond to $n$, $h$, $f$, $f_n$, $I_n$ in \cite{horvath1991lp}. We do not present the exact expressions for $\sigma(N)$ and $e(N)$; they are written in \cite[page 1934]{horvath1991lp}. The uniform boundedness of $\sigma$ is exactly line (1.2) of \cite{horvath1991lp};  the  boundedness for $e(N)$ is not written explicitly in \cite{horvath1991lp} so we briefly explain  how to derive it. In the expression for $e(N)$ in \cite{horvath1991lp}, $m_N$ is the only term that depends on $N$. A standard argument that uses the Taylor expansion of $\rho$ and the symmetry of the kernel $K$ (see, for example, Section 2.4 of the lecture notes \cite{hansen2009lecture}) yields that $m_N$ is uniformly bounded in $N$.  
\end{proof}
From this it is easy to deduce:
\begin{cor}
Under the hypotheses of Theorem \ref{thm:H}, the error $e_N^{\delta(N)}$ converges in distribution to zero as $N\rightarrow \infty$.
\end{cor}

\begin{rem}
\label{rem:smooth} 
There are a few ways in which practical situations may not align perfectly with the assumptions of \cite{horvath1991lp}. However, we posit that in all of these cases, the difference between these situations and that studied in \cite{horvath1991lp} is numerically insignificant. We now briefly summarize these three discrepancies and indicate how to resolve them. 

First, we defined our density $\rho_N^\delta$ by (\ref{eqn:normalized_blob_function}), but in this section we use a version with denominator $N$. However, as explained above, the two expressions approach each other for small $\delta$, and this is the situation we are interested in here. Second, a $\rho$ that is piecewise continuous like the ring distribution is not twice differentiable. We point out that an arbitrary density $\rho$ may be approximated to arbitrary precision by a smoothed out version, for example by convolution with a mollifier (a standard reference is Brezis \cite[Section 4.4]{brezis2010functional}).  Third, in our computations we use the kernel $G$, which is not compactly supported, for the sake of simplicity. Similarly, this kernel can be approximated, with arbitrary accuracy, by a compactly supported version. Making these changes to the kernel or target density would not affect the conclusions of numerical results.
\end{rem}

\subsection{Numerical Approximation of the Error Metric PDF}
\label{subsec:num}
In this subsection we describe how to numerically find $f_{e_N^\delta}$ and $F_{e_N^\delta}$.  
For sufficiently large $N$, one could simply use random sampling to estimate the mean and standard deviation, then take these as the parameters of the normal PDF (i.e. the error function and Gaussian function, respectively). However, for moderate $N$, we choose to begin by estimating the entire CDF and confirming that it is approximately normal. We first establish:  
\begin{prop} We have, 
\begin{equation}
F_{e_N^\delta}(z)=\int_{\Omega^N} \boldsymbol{1}_{\{ x\mid e_N^{\delta}(x)\le z \}} \prod^N_{i=1} \rho( x_i)  dx. \label{Derived Distribution}
\end{equation}
\end{prop}
\begin{proof}
We recall a basic probability fact. Let $Y$ be a random vector with values in $A\subset \RR^D$  with probability density function $f$, and let $g$ be a real-valued function on $\RR^d$. The CDF for $g(Y)$, denoted $F_{g(Y)}$,  is given by,
\begin{equation}
\label{cdf}
F_{g(Y)}(z) = \mathbb{P}(g(Y) \leq z) = \int_{A} \boldsymbol{1}_{\{y| g(y) \leq z\}} f(y) dy,
\end{equation}
where $\boldsymbol{1}$ denotes the indicator function. 

In our situation, we take the random vector $Y$ to be $X:=(X_1, ..., X_N)$. Since $X$ takes values in $\Omega^N:=\Omega\times...\times\Omega$, we take $A$ to be $\Omega^N$ (we point out that here $D=2N$). Since each $X_i$ has density $\rho$, the density of $X$ is the function $\tilde{\rho}$, defined by,
\[
\tilde{\rho}(x_1, ..., x_n) = \prod^N_{i=1} \rho( x_i).
\] 
Thus, taking $f$ and $g$ in (\ref{cdf}) to be $\tilde{\rho}$ and  $e_N^\delta$, respectively, yields (\ref{Derived Distribution}).  
\end{proof}

Notice that since each of the $x_i$ is itself a 2-dimensional vector (the $X_i$ are random points in the plane and we are using the notation $x=(x_1,..., x_N)$),  the integral defining the cumulative distribution function of the error metric is of dimension $2N$. Finding analytical representations for the CDF is combinatorially complex and quickly becomes infeasible for large swarms. Therefore, we approximate  (\ref{Derived Distribution}) using Monte Carlo integration, which is well-suited for high-dimensional integration \cite{sloan2010integration}\footnote{Quasi-Monte Carlo techniques, which use a low-discrepancy sequence rather than truly random evaluation points, promise somewhat faster convergence but require considerably greater effort to implement. The difficulty is in generating a low-discrepancy sequence from the desired distribution, which is possible using the Hlawka-M\"uck method, but computationally expensive \cite{hartinger2006non}.}, and fit a Gauss error function to the data. If the fitted curve matches the data well, we differentiate to obtain the PDF. We remark that we have used the notation of an indicator function above in order to express the quantity of interest in a way that is easily approximated with Monte Carlo integration.

\subsection{Example}
\label{subsec: num ex}

We apply this method to assess the performance of the controller in \cite{li2017decentralized}, again for the ``ring distribution'' scenario with $N=200$ robots described in Section \ref{example optimization}.


We approximate  $F_{e_N^\delta}$ using $M=1000$ Monte Carlo evaluation points; this is shown by a solid gray line in Figure \ref{fig:error_probabilities}. The numerical approximation appears to closely match a Gauss error function ($\erf(\cdot)$, the integral of a Gaussian $G$) as theory predicts. Therefore an analytical $\erf(\cdot)$ curve, represented by the dashed line, is fit to the data using MATLAB's least squares curve fitting routine \texttt{lsqcurvefit}. To obtain $f_{e_N^\delta}$, the analytical curve fit for $F_{e_N^\delta}$ is differentiated, and the result is also shown in Figure \ref{fig:error_probabilities}.

\begin{figure}[ht]
\includegraphics[width=\linewidth]{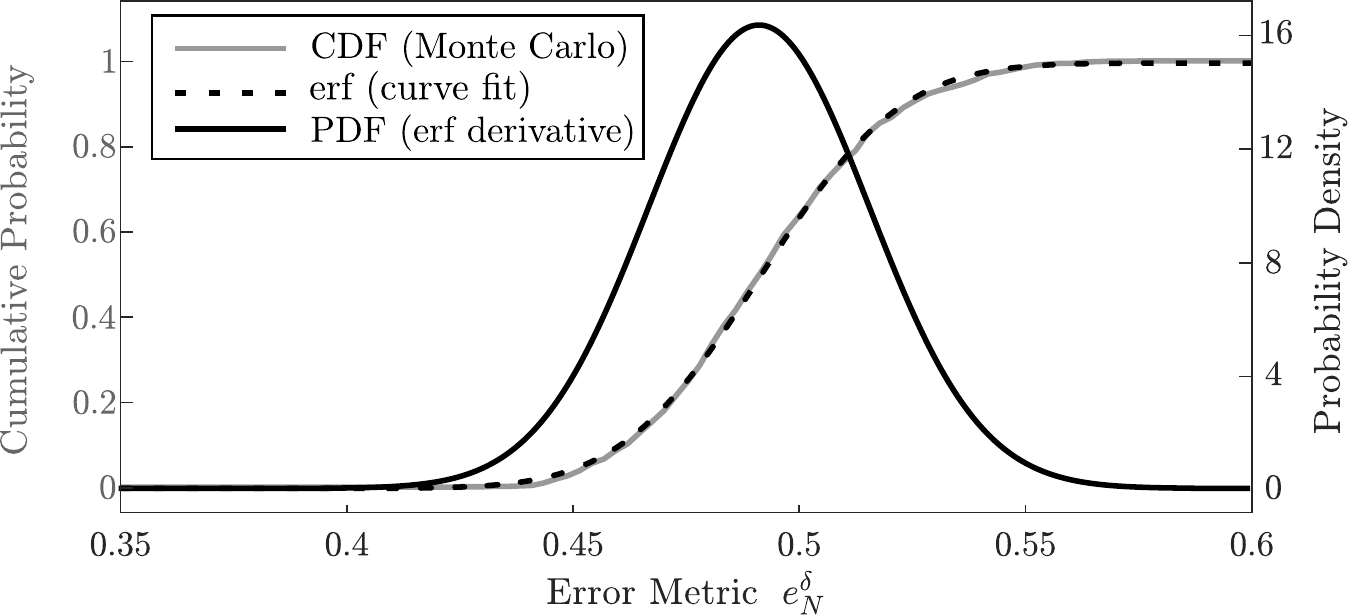}
\caption{The CDF of the error metric when robot positions are sampled from $\rho$ is approximated by Monte Carlo integration, an $\erf$ curve fit matches closely, and the PDF is taken as the derivative of the fitted $\erf$.}
\label{fig:error_probabilities}
\end{figure}

With the error metric distribution now confirmed to be approximately normal, the F- and T-tests \cite{moore2009introduction} are appropriate statistical procedures for comparing the steady state error distribution to $f_{e_N^\delta}$.

From data presented as Figure 7 of \cite{li2017decentralized}, we calculate the distribution of steady state error metric values produced by the controller to have a mean of $0.5026$ with a standard deviation of $0.02586$. We take the null hypothesis to be that the distribution of these error metric values is the same as $f_{e_N^\delta}$, which has a sample mean of 0.4933 and a standard deviation of 0.02484, as calculated from the $M = 1000$ samples. A two-sample F-test fails to refute the null hypothesis, with an F-statistic of 1.0831, indicating no significant difference in the standard deviations. On the other hand, a two-sample T-test rejects the null hypothesis with a T-statistic of 8.5888, indicating that the steady state error is not distributed with the same population mean as $f_{e_N^\delta}$. However, the 95\% confidence interval for the true difference between population means is computed to be $(0.00717,0.01141)$, showing that the mean steady state error achieved by this controller is unlikely to exceed that of $f_{e_N^\delta}$ by more than 2.31\%. Therefore, we find the performance of the controller in \cite{li2017decentralized} to be acceptable given its stochastic nature, as the error metric values it realizes are only slightly different from those produced by sampling robot positions from the target distribution.

As with $e_\text{rel}$ of Section \ref{sec:extrema}, the sentiment of this benchmark is preceded by \cite{li2017decentralized}. However, without prior knowledge that $f_{e_N^\delta}$ would be approximately Gaussian, the  calculation took two orders of magnitude more computation in that study\footnote{According to the caption of Figure 2 of \cite{li2017decentralized}, the figure was generated as a histogram from 100,000 Monte Carlo samples.}. Also, where it is noted in \cite{li2017decentralized} that ``the error values [from simulation] mostly lie between \dots the 25th and 75th percentile error values when robot configurations are randomly sampled from the target distribution'', we have replaced visual inspection with the appropriate statistical tests for comparing two approximately normal distributions. These improvements make this error metric PDF benchmark objective and efficient, and thus suitable for common use.

\section{\uppercase{Future Work and Conclusion}}
\label{sec:conclusion}

\noindent While the concepts presented herein are expected to be sufficient for the comparison and evaluation of swarm distribution controllers, the computational techniques are certainly open to analysis and improvement. For instance, is there a simpler, deterministic method of approximating the error metric PDF $f_{e_N^\delta}$? Is there a more appropriate formulation for determining the extrema of the error metric for a given situation, one that is guaranteed to produce a global optimum? If not, which nonlinear programming algorithm is most suitable for solving (\ref{Optimization Problem})? In practice, what method of quadrature converges most efficiently to approximate the error metric? As the size of practical robot swarms will likely grow faster than processor speeds will increase, improved computational techniques will be needed to keep benchmark computations practical.

Also, a very important question remains about the nature of the error metric.  The blob shape $K^\delta$ has an intuitive physical interpretation, and so a reasonable choice is typically easy to make. The value of the error metric for a particular situation is certainly affected by the choice of blob shape $K$ and radius $\delta$, but so are the values of the proposed benchmarks: the extrema and the error PDF. Are qualitative conclusions made by comparing the performance of the controller to these benchmarks likely to be affected by the choice of blob? 

Open questions notwithstanding, the error metric presented herein is sufficiently general to be used in quantifying the the performance of one of the most fundamental tasks of a robotic swarm controller: achieving a prescribed density distribution. The error metric is sensitive enough to compare the effectiveness of given control laws for achieving a given target distribution. The error metric parameters, blob shape and radius, have  intuitive physical interpretations so that they can be chosen appropriately. Should a designer wish to interpret the performance of a given controller without comparing against results of another controller, we provide two benchmarks that can be applied to any situation: extrema of the error metric, and the  probability density function of the error metric when swarm configurations are sampled from the target distribution. Using the provided code, these methods can easily be used to quantitatively assess the performance of new swarm controllers and thereby improve the effectiveness of practical robot swarms.

\section*{\uppercase{Acknowledgements}}

\noindent The authors gratefully acknowledge the support of NSF grant CMMI-1435709, NSF grant DMS-1502253, the Dept. of Mathematics at UCLA, and the Dept. of Mathematics at Harvey Mudd College. The authors also thank Hao Li (UCLA) for his insightful contributions regarding the connection between this error metric and kernel density estimation.

\bibliographystyle{apalike}
{\small
\bibliography{ICINCO_2018}}

\begin{thebibliography}{}

\bibitem[Ayvali et~al., 2017]{ayvali2017ergodic}
Ayvali, E., Salman, H., and Choset, H. (2017).
\newblock Ergodic coverage in constrained environments using stochastic
  trajectory optimization.
\newblock In {\em Intelligent Robots and Systems (IROS), 2017 IEEE/RSJ
  International Conference on}. IEEE.

\bibitem[Barbosa et~al., 2004]{barbosa2004tuning}
Barbosa, R.~S., Machado, J.~T., and Ferreira, I.~M. (2004).
\newblock Tuning of {PID} controllers based on {B}ode's ideal transfer
  function.
\newblock {\em Nonlinear dynamics}, 38(1):305--321.

\bibitem[Berman et~al., 2011]{berman2011design}
Berman, S., Kumar, V., and Nagpal, R. (2011).
\newblock Design of control policies for spatially inhomogeneous robot swarms
  with application to commercial pollination.
\newblock In {\em Robotics and Automation (ICRA), 2011 IEEE International
  Conference on}, pages 378--385. IEEE.

\bibitem[Brambilla et~al., 2013]{brambilla2013swarm}
Brambilla, M., Ferrante, E., Birattari, M., and Dorigo, M. (2013).
\newblock Swarm robotics: a review from the swarm engineering perspective.
\newblock {\em Swarm Intelligence}, 7(1):1--41.

\bibitem[Brezis, 2010]{brezis2010functional}
Brezis, H. (2010).
\newblock {\em Functional analysis, Sobolev spaces and partial differential
  equations}.
\newblock Springer Science \& Business Media.

\bibitem[Bruemmer et~al., 2002]{bruemmer2002robotic}
Bruemmer, D.~J., Dudenhoeffer, D.~D., McKay, M.~D., and Anderson, M.~O. (2002).
\newblock A robotic swarm for spill finding and perimeter formation.
\newblock Technical report, Idaho National Engineering and Environmental Lab,
  Idaho Falls.

\bibitem[Cao et~al., 1997]{cao1997cooperative}
Cao, Y.~U., Fukunaga, A.~S., and Kahng, A. (1997).
\newblock Cooperative mobile robotics: Antecedents and directions.
\newblock {\em Autonomous robots}, 4(1):7--27.

\bibitem[Cortes et~al., 2004]{Cortes}
Cortes, J., Martinez, S., Karatas, T., and Bullo, F. (2004).
\newblock Coverage control for mobile sensing networks.
\newblock {\em IEEE Transactions on Robotics and Automation}.

\bibitem[Craig and Bertozzi, 2016]{CB}
Craig, K. and Bertozzi, A. (2016).
\newblock A blob method for the aggregation equation.
\newblock {\em Mathematics of computation}, 85(300):1681--1717.

\bibitem[Demir et~al., 2015]{demir2015decentralized}
Demir, N., Eren, U., and A{\c{c}}{\i}kme{\c{s}}e, B. (2015).
\newblock Decentralized probabilistic density control of autonomous swarms with
  safety constraints.
\newblock {\em Autonomous Robots}, 39(4):537--554.

\bibitem[Devroye and Gy\H{o}rfi, 1985]{DG}
Devroye, L. and Gy\H{o}rfi, L. (1985).
\newblock {\em Nonparametric density estimation: the $L^1$ view}.
\newblock Wiley.

\bibitem[Elamvazhuthi et~al., 2016]{elamvazhuthi2016coverage}
Elamvazhuthi, K., Adams, C., and Berman, S. (2016).
\newblock Coverage and field estimation on bounded domains by diffusive swarms.
\newblock In {\em Decision and Control (CDC), 2016 IEEE 55th Conference on},
  pages 2867--2874. IEEE.

\bibitem[Elamvazhuthi and Berman, 2015]{elamvazhuthi2015optimal}
Elamvazhuthi, K. and Berman, S. (2015).
\newblock Optimal control of stochastic coverage strategies for robotic swarms.
\newblock In {\em Robotics and Automation (ICRA), 2015 IEEE International
  Conference on}, pages 1822--1829. IEEE.

\bibitem[Hamann and W{\"o}rn, 2006]{hamann2006analytical}
Hamann, H. and W{\"o}rn, H. (2006).
\newblock An analytical and spatial model of foraging in a swarm of robots.
\newblock In {\em International Workshop on Swarm Robotics}, pages 43--55.
  Springer.

\bibitem[Hansen, 2009]{hansen2009lecture}
Hansen, B.~E. (2009).
\newblock Lecture notes on nonparametrics.
\newblock {\em Lecture notes}.

\bibitem[Hartinger and Kainhofer, 2006]{hartinger2006non}
Hartinger, J. and Kainhofer, R. (2006).
\newblock Non-uniform low-discrepancy sequence generation and integration of
  singular integrands.
\newblock In {\em Monte Carlo and Quasi-Monte Carlo Methods 2004}, pages
  163--179. Springer.

\bibitem[Horv\'ath, 1991]{horvath1991lp}
Horv\'ath, L. (1991).
\newblock On {$L_p$}-norms of multivariate density estimators.
\newblock {\em The Annals of Statistics}, pages 1933--1949.

\bibitem[Howard et~al., 2002]{howard2002mobile}
Howard, A., Mataric, M.~J., and Sukhatme, G.~S. (2002).
\newblock Mobile sensor network deployment using potential fields: A
  distributed, scalable solution to the area coverage problem.
\newblock {\em Distributed autonomous robotic systems}, 5:299--308.

\bibitem[Hussein and Stipanovic, 2007]{HS}
Hussein, I.~I. and Stipanovic, D.~M. (2007).
\newblock Effective coverage control for mobile sensor networks with guaranteed
  collision avoidance.
\newblock {\em IEEE Transactions on Control Systems Technology},
  15(4):642--657.

\bibitem[Li et~al., 2017]{li2017decentralized}
Li, H., Feng, C., Ehrhard, H., Shen, Y., Cobos, B., Zhang, F., Elamvazhuthi,
  K., Berman, S., Haberland, M., and Bertozzi, A.~L. (2017).
\newblock Decentralized stochastic control of robotic swarm density: Theory,
  simulation, and experiment.
\newblock In {\em Intelligent Robots and Systems (IROS), 2017 IEEE/RSJ
  International Conference on}. IEEE.

\bibitem[Moore et~al., 2009]{moore2009introduction}
Moore, D.~S., McCabe, G.~P., and Craig, B.~A. (2009).
\newblock {\em Introduction to the Practice of Statistics}.
\newblock W. H. Freeman.

\bibitem[Reif and Wang, 1999]{reif1999social}
Reif, J.~H. and Wang, H. (1999).
\newblock Social potential fields: A distributed behavioral control for
  autonomous robots.
\newblock {\em Robotics and Autonomous Systems}, 27(3):171--194.

\bibitem[Schwager et~al., 2006]{schwager2006distributed}
Schwager, M., McLurkin, J., and Rus, D. (2006).
\newblock Distributed coverage control with sensory feedback for networked
  robots.
\newblock In {\em robotics: science and systems}.

\bibitem[Shen et~al., 2004]{shen2004hormone}
Shen, W.-M., Will, P., Galstyan, A., and Chuong, C.-M. (2004).
\newblock Hormone-inspired self-organization and distributed control of robotic
  swarms.
\newblock {\em Autonomous Robots}, 17(1):93--105.

\bibitem[Sloan, 2010]{sloan2010integration}
Sloan, I. (2010).
\newblock Integration and approximation in high dimensions---a tutorial.
\newblock {\em Uncertainty Quantification, Edinburgh}.

\bibitem[Soysal and {\c{S}}ahin, 2006]{soysal2006macroscopic}
Soysal, O. and {\c{S}}ahin, E. (2006).
\newblock A macroscopic model for self-organized aggregation in swarm robotic
  systems.
\newblock In {\em International Workshop on Swarm Robotics}, pages 27--42.
  Springer.

\bibitem[Spears et~al., 2004]{spears2004distributed}
Spears, W.~M., Spears, D.~F., Hamann, J.~C., and Heil, R. (2004).
\newblock Distributed, physics-based control of swarms of vehicles.
\newblock {\em Autonomous Robots}, 17(2):137--162.

\bibitem[Sugihara and Suzuki, 1996]{sugihara1996distributed}
Sugihara, K. and Suzuki, I. (1996).
\newblock Distributed algorithms for formation of geometric patterns with many
  mobile robots.
\newblock {\em Journal of Field Robotics}, 13(3):127--139.

\bibitem[Zhang et~al., 2018]{zhang2017performance}
Zhang, F., Bertozzi, A.~L., Elamvazhuthi, K., and Berman, S. (2018).
\newblock Performance bounds on spatial coverage tasks by stochastic robotic
  swarms.
\newblock {\em IEEE Transactions on Automatic Control}, 63(6):1473--1488.

\bibitem[Zhong and Cassandras, 2011]{ZC}
Zhong, M. and Cassandras, C.~G. (2011).
\newblock Distributed coverage control and data collection with mobile sensor
  networks.
\newblock {\em IEEE Transactions on Automatic Control}, 56(10):2445--2455.

\end{thebibliography}
\end{document}